\documentclass[12pt,a4paper,reqno]{amsart}
\usepackage[english]{babel}
\usepackage[utf8]{inputenc}
 \usepackage{hyperref}
\allowdisplaybreaks  

\usepackage[textsize=Small, backgroundcolor=orange!80, disable]{todonotes}  
\usepackage{amssymb}
\usepackage{tikz-cd} 

\newtheorem{thm}[equation]{Theorem}
\newtheorem*{thm*}{Theorem}

\newtheorem{cor}[equation]{Corollary}
\newtheorem{lem}[equation]{Lemma}
\newtheorem{prop}[equation]{Proposition}

\theoremstyle{remark}
\newtheorem{rem}[equation]{Remark}
\theoremstyle{remark}
\newtheorem{exam}{Example}

\numberwithin{equation}{section} 

 \newcommand{\Hom}{\operatorname{Hom}}
\newcommand{\ad}{\operatorname{ad}}
 \newcommand{\Witt}{\operatorname{Witt}}
 \newcommand{\Diff}{\operatorname{Diff}}
 \newcommand{\Vir}{\operatorname{Vir}}
 \newcommand{\ord}{\operatorname{ord}}
 \newcommand{\SGL}{\operatorname{SGL}}
  \newcommand{\sop}{\operatorname{supp}}

 \renewcommand\sl{\mathfrak{sl}}
 \newcommand\g{{\mathfrak{g}}}

\newcommand\Czz{{\mathbb C}[\mkern-.42\thinmuskip[z]\mkern-.45\thinmuskip]}
\newcommand\C{{\mathbb C}}

\newcommand\Z{{\mathbb Z}}
\newcommand\Cz{\C(\!(z)\!)}

\title[Lie subalgebras of Differential Operators]{Lie subalgebras of Differential Operators in one Variable}

\author[F. J. Plaza Mart\'{\i}n]{F. J. Plaza Mart\'{\i}n}
\author[C. Tejero Prieto]{C. Tejero Prieto}

\address{Departamento de Matem\'aticas and IUFFyM, Universidad de
Salamanca,  Plaza de la Merced 1-4
        \\
        37008 Salamanca. Spain.
        \\
         Tel: +34 923294460. Fax: +34 923294583}
\thanks{
\indent This work is supported by the research contracts MTM2017-86042-P and
MTM2015-66760-P of MINECO (Spain) and FS/29-2017 Fundación Solórzano.  \\
}
\email{fplaza@usal.es}
\email{carlost@usal.es}

\begin{document}

\begin{abstract} 
Let $\Witt$ be the Lie algebra generated by  the  set $\{L_i\,\vert\, i \in {\mathbb Z}\}$ and $\Vir$ its universal central extension. Let $\Diff(V)$ be the Lie algebra of differential operators on $V=\Cz$, $\Czz$  or $V=\C(z)$. We explicitly describe all Lie algebra homomorphisms from $\sl(2)$, $\Witt$ and $\Vir$ to $\Diff(V)$ such that $L_0$ acts on $V$ as a first order differential operator. 
\end{abstract}

 \maketitle

\section{Introduction}

The study of Lie algebras of differential operators is a crucial step in several problems such as Conformal Field Theory or Gromov-Witten theory. Indeed, the description of the representations of a Lie algebra by vector fields is a classical problem that was first considered by S. Lie himself (\cite{Lie}). A natural generalization of this problem deals with the classification of realizations in terms of differential operators. It is worth pointing out that the explicit expressions of these representations can be a powerful tool to tackle a variety of problems.  Let us give some evidences supporting this claim.

In the case of finite dimensional Lie algebras, a better knowledge of their representations is of great help in a variety of problems, such as, integration of ordinary differential equations, group classification of partial differential equations,  classification  of  gravity  fields  of  a  general  form  under   motion  groups, geometric control theory, Levine's problem, etc. See, for instance,  \cite{Draisma,Kamran-Olver,Popovych,Turbiner-sl, Turbiner}  and the references therein.


Regarding the case of infinite dimensional Lie algebras, during the last decades some Lie algebras (Virasoro, Witt, Krichever-Novikov, etc.) have emerged as highly relevant objects in  mathematics and in mathematical physics. As an illustration, it is enough to mention the role played by the  Virasoro algebra in Conformal Field Theory, Eynard-Orantin topological recursion, Virasoro conjecture, etc. For further properties and applications on this algebra, and with no aim to be exhaustive, let us mention \cite{Iohara-Koga, Kac} as well as \cite{Frenkel,Mulase,Plaza-Opers, Plaza-Sigma,Pope,VirKdV}.

Now, let us be more precise. Let $\Witt$ be the complex Lie algebra generated by $\{L_i\,\vert\, i\in {\mathbb Z}\}$ and whose  bracket is $[L_i, L_j] :=(i-j)L_{i+j}$. Let $\Witt_>$ and $\Witt_<$ be the Lie-subalgebras generated by $\{L_i\,\vert\, i\geq -1\}$ and $\{L_i\,\vert\, i \leq 1\}$ respectively. Let  ${\Vir}$ be the Virasoro algebra; that is, the central extension of the $\Witt$ algebra associated to the cocycle:
	\begin{equation}\label{eq:Vir-co}
	\Psi(L_i,L_j)\, :=\, \frac1{12}(i^3-i) \delta_{i+j,0}\, .
	\end{equation} 
Note that the subalgebra $\langle L_{-1}, L_0,L_1\rangle$ is isomorphic to the Lie algebra $\sl(2)$ and that, by fixing such an isomorphism, we get inclusions:
$$ \begin{tikzcd} &\ \Witt_<\quad \ar[dr,hook] &\\  \sl(2)\ar[ur,hook] \ar[dr,hook] & & \Vir\\ & \ \Witt_>\ar[ur,hook] &  \end{tikzcd} $$
These chains of subalgebras will be intensively used along this paper since a representation of $\Witt_>$ (resp. $\Witt_<$) will be thought of as one of $\sl(2)$ admitting an extension to $\Witt_>$ (resp. $\Witt_<$), and similarly for $\Vir$ with respect to $\Witt_>$ and $\Witt_<$.  Actually, this strategy  is deeply influenced by \cite{Plaza-Tejero} where the question of whether an $\sl(2)$-module admits a compatible structure of $\Witt_>$-module (resp. $\Witt_<$-module) has been solved in full generality.

The main result of this paper is an explicit description and classification of the following set of realizations:
	\[
	\big\{ \rho\in \Hom_{\text{Lie-alg}}({\mathfrak{g}}, \Diff(V)) \text{ s.t. $\rho(L_0)$ is of first order} \big\} \, ,
	\]
for certain choices of ${\mathfrak{g}}$ and $V$; namely, the cases ${\mathfrak{g}}=\sl(2)$, $\Witt_>$, $\Witt_<$, $\Witt$ or $\Vir$ and $V=\C(z), \Czz$ or $ \Cz$ will be exhaustively studied. 

Let us mention in passing that  the constraint that $\rho(L_0)$ is a first order differential operator is granted in many interesting physical applications. For instance, in Conformal Field Theory, the operator $\rho(L_0)$ corresponds to the dilaton equation in string theory which is always of first order (e.g. \cite{Ribault}). On the mathematical side, this constraint on the order also appears in many problems as in the Virasoro constraints and Gromov-Witten theory (e.g. \cite{Givental}), representation theory of $\sl(2)$ (see \cite{Block, Turbiner, Turbiner-sl}) and of ${\mathcal W}_{\infty}$-algebras (\cite{Frenkel, Pope}) and the study of the KP hierarchy (\cite{Mulase, Plaza-Sigma}). Moreover, the fact that $\rho(L_0)$ is of first order is also relevant in the finite dimensional situation; it is worth mentioning the case of Levine's problem (\cite{Kamran-Olver}).

Let us explain the contents of the paper. 

After proving some technical facts in \S2, in \S\ref{sec:expre} we focus on the case of representations of $\sl(2)$ in $\Diff(V)$ for $V=\C(z), \Czz, \Cz$ where $L_0$ is realized as a first order differential operator. All these representations are completely classified and explicit expressions are provided (see Theorem~\ref{t:repr-sl2} for the cases $\C(z), \Cz$ and Theorems~\ref{t:sl2-C[[z]]} and~\ref{t:sl2-line} for $ \Czz$). Let us also mention two remarkable facts about these representations. The first one is that the Casimir operator of any representation of this type acts by a constant. The second one is that, through the study of the enveloping algebra, our results on subalgebras of differential operators are closely related to Block's original approach to irreducible $\sl(2)$-modules by means of differential operators \cite{Block} and the interpretation given by Bavula \cite{Bavula}. 

Once the case of $\sl(2)$ has been solved, in \S\ref{sec:Witt}  we address the problem for the Lie algebras $\Witt_>$, $\Witt_<$, $\Witt$ and $\Vir$. Again, Theorem~\ref{t:Witt>} solves  completely the question for the vector spaces $\C(z), \Cz$ and Theorem~\ref{t:Witt>-line} for $\Witt_>$ and $ \Czz$.  It is also studied how these representations restrict to $\sl(2)$. In particular, our results of \S\ref{sec:expre}-\S\ref{sec:Witt} unveil the intimate connection between the representation theory of $\sl(2)$ and that of ${\mathfrak g}=\Witt_>, \Witt_<,\Witt,\Vir$.  Note that all results stated for $\Witt_>$ and $\C[[z]]$ are also valid for $\Witt_<$ via the Chevalley involution.

%

The last section, \S\ref{sec:remarks}, begins with some facts on the universal enveloping algebra of certain representations. Then, we  succinctly recall several instances of representations of $\sl(2)$ in terms of differential operators, already present in the literature. It will be shown that all of these instances fit into our setup (\cite{Block, Givental, Iohara-Koga, Mulase,Ribault, Turbiner}).


As a future research, it would also be interesting to obtain explicit representations of $\Witt_>$ or $\Witt_<$ by means of the theory of irreducible representations of $\sl(2)$, see \cite{Block, JAlg}. On the other hand, it is reasonable to expect some new results in  the study of the representation theory of the algebra ${\mathcal D}$ of differential operators on the circle and its central extension $\widehat{\mathcal D}$ (a.k.a. ${\mathcal W}_{1+\infty}$) as a consequence of our results of \S\ref{subsec:env-alg-w}. 

\textbf{Acknowlegement:} The authors would like to thank the anonymous referee for the careful reading of the paper and pointing out many typos and some inaccuracies that have helped to greatly improve this work.

%

\section{Symbols and Orders}\label{sec:symbolorder}

In this section, we set  $V=\C(z), \Cz$ and ${\mathfrak{g}}$ is one of the following complex Lie algebras: $\sl(2), \Witt, \Witt_>, \Witt_<$. Recall that there is a standard choice for a basis of $\sl(2)$;  namely, the so-called  Chevalley basis consisting of a triple $\{e,f,h\}$  where $[e,f]=h$, $[h,e]=2e$ and $[h,f]=-2f$. Hence, mapping $f$ to $L_{-1}$, $h$ to $-2L_0$  and $e$ to $-L_1$ yields and identification $\sl(2) =\langle L_{-1}, L_0, L_1\rangle = \Witt_>\cap \Witt_<$; in particular, 
	\[
	\sl(2)=\langle L_{-1}, L_0, L_1\rangle \,\hookrightarrow\, \Witt_>=\langle \{L_i\,\vert\, i\geq -1 \}\rangle\, ,  
	\] 
with Lie bracket $[L_i,L_j]=(i-j)L_{i+j}$ for all $L_i\in{\mathfrak{g}}$.  The support of $\mathfrak g$ is the subset of $\Z$ given by $\sop(\mathfrak g)=\{i\in \Z\colon L_i\in\mathfrak g\}$. 

The Chevalley involution $\Theta\colon \Witt\to \Witt$, defined  by
	\begin{equation}\label{eq:Chevalley}
	\Theta(L_i)=(-1)^{i+1} L_{-i}\, , 
	\end{equation}
defines Lie algebra isomorphisms $\Theta\colon \Witt_>\xrightarrow{\sim}\Witt_<$, $\Theta\colon \Witt_<\xrightarrow{\sim}\Witt_>$ and $\Theta\colon \sl(2)\xrightarrow{\sim} \sl(2)$.

The Casimir of $\sl(2)$, see \cite[Lemma 1.30, page 11]{Mazorchuk}, is the element of its universal enveloping algebra given by the equivalent expressions 
	\begin{align}\label{eq:Casimir1}
	\mathcal C&=(h+1)^2+4fe=4\left(\left(L_0-\frac{1}{2}\right)^2-L_{-1}L_1\right)=\\
	\label{eq:Casimir2}&=(h-1)^2+4ef=4\left(\left(L_0+\frac{1}{2}\right)^2-L_{1}L_{-1}\right).
	\end{align} Therefore, one has $\Theta(\mathcal C)=\mathcal C$.
Let us recall that differential operators $P\in \Diff(V)$ are finite linear combinations $$P=\sum_{j=0}^k \xi_j\, \partial^j$$where $\partial$ denotes $\frac{\partial}{\partial z}$ and $\xi_i\in V$. If $\xi_k\neq 0$ then we say that $P$ has order $k$ and we denote by $\Diff^k(V)\subset \Diff(V)$ the subset of $k$-th order differential operators. The commutator endows $\Diff(V)$ with a Lie algebra structure and one has $[\Diff^k(V),\Diff^l(V)]\subset \Diff^{k+l-1}(V)$.

Assume  that a representation $\rho: {\mathfrak{g}}\to \Diff(V)$ is given. Since ${\mathfrak{g}}$ is simple, it follows that $\rho$ is either zero or injective. Thus, from now on we assume that $\rho$ is injective. We consider the symbol, $a_i\in V$, and the order, $n_i\in {\mathbb Z}$ with $n_i\geq 0$, of the differential operator $\rho(L_i)$; that is,
	\[
	\rho(L_i)\, =\, a_i \partial^{n_i} + \text{lower order terms}\, \in \, \Diff^{n_i}(V), 
	\]
where $a_i\neq 0$. From now on, we define $a'$ to be $\partial a$ where $a\in V$. We denote by $\ord(L)$ the order of a differential operator $L$; and, thus, $n_i:=\ord(\rho(L_i))$.

 A key observation is that the identity $[\rho(L_i), \rho(L_j)]=(i-j)\rho(L_{i+j})$ implies the following relation:
	\begin{equation}\label{e:ci>1}
	n_i+n_j -1 \,\geq \, n_{i+j}, \quad\text{for all}\  i\neq j. 
	\end{equation}

More precisely, we have  the following:
\begin{prop}\label{p:fundamental}
For $i \neq j$, the following  conditions are equivalent:
\begin{enumerate}
	\item $n_i+n_j -1> n_{i+j}$, 
	\item the coefficient of $\partial^{n_i+n_j -1}$ of $[\rho(L_i), \rho(L_j)]$ vanishes, that is $$n_i a_i a'_j- n_j a_j a'_i =0.$$
	\item  $a_i^{n_j}$ is equal to $a_j^{n_i}$ up to a non zero multiplicative constant.
\end{enumerate}
\end{prop}

\begin{proof}
Having in mind~\eqref{e:ci>1}, we  consider three cases. First, if $n_i=n_j=0$, then $\rho(L_{i+j})=0$ and this is a contradiction since $\rho$ is injective. Second, if $n_i,n_j>0$, then  the coefficient of $\partial^{n_i+n_j -1}   $ in $[\rho(L_i), \rho(L_j)]$ is given by  $n_i a_i a'_j- n_j a_j a'_i $,  thus the conclusion follows easily.  The remaining case, $n_i=0,n_j>0$ is similar.
\end{proof}

\begin{cor}\label{cor:fundamental}
For $i \neq j$, the following  conditions are equivalent:
\begin{enumerate}
	\item $n_i+n_j -1= n_{i+j}$, 
	\item the coefficient of $\partial^{n_i+n_j -1}$ of $[\rho(L_i), \rho(L_j)]$ coincides with the symbol of $(i-j)\rho(L_{i+j})$, that is $$n_i a_i a'_j- n_j a_j a'_i =(i-j)a_{i+j}.$$
	\item  $n_i a_i a'_j- n_j a'_i a_j\neq 0.$
\end{enumerate}
\end{cor}


\begin{prop}\label{p:a_ia_1a_0} 
If $n_0= 1$, then for every $i,j\in\sop(\g)$, one has

  \begin{equation}\label{eq:clave}
	 a_0 a'_i - n_i a_ia'_0  \, =\, (-i) a_i,
	\end{equation}  
\begin{equation}\label{eq:clave2}
a_0(n_ia_ia_j'-n_ja_ja_i')=(in_j-jn_i)a_ia_j,\end{equation}
\begin{equation}\label{eq:clave3}
a_0(ia_ia_j'-ja_ja_i')=(in_j-jn_i)a_ia_ja_0'.
\end{equation}
Moreover,  for every $k\in \sop(\g)$ there exist non vanishing constants $\mu_k \in \C^*$, with $\mu_0=\mu_1=1$, such that  $$a_k=\mu_k a_1^ka_0^{n_k-k n_1}.$$
\end{prop}

\begin{proof}
Under the assumption $n_0=1$, it holds that  $n_0+n_k-1= n_k$. Therefore, by Corollary \ref{cor:fundamental} we get the equality:
	$$
	a_0 a'_k - n_k a_k a'_0 \, =\, (-k) a_k.
	$$ 
which is \eqref{eq:clave}. Taking  this equation for $k=j$ multiplied by $n_ia_i$ (resp. $ia_i$) minus the same equation for $k=i$ multiplied by $n_ja_j$ (resp. $ja_j$) we get the equality 
	\begin{align*}
	a_0(n_ia_ia_j'-n_ja_ja_i')=(in_j-jn_i)a_ia_j,
	\\ (\text{resp.}\ a_0(ia_ia_j'-ja_ja_i')=(in_j-jn_i)a_ia_ja_0'). \end{align*} 
and, thus, equations \eqref{eq:clave2} and \eqref{eq:clave3} are proved. 

In a similar way, 
considering equation (\ref{eq:clave3}) for $i=k$, $j=1$ and dividing it by $a_k a_1 a_0$, we obtain the differential equation:
	\[
	 \frac{a'_k}{a_k}- k  \frac{a'_1}{a_1}  - (n_k- k n_1) \frac{a'_0}{a_0}\, =\, 0
	\,. \]
 Integrating it, the last claim follows and clearly  $\mu_0=\mu_1=1$.
\end{proof}

\begin{prop}\label{p:n_i=}
If $n_0= 1$, then $n_i+n_j-1= n_{i+j}$ for  all $i\neq j$ such that $-1\leq i,j\leq 1$.
\end{prop}

\begin{proof} Thanks to Corollary \ref{cor:fundamental}, equation (\ref{eq:clave}) implies the claim when either $i=0$ or $j=0$. Hence, we only need to consider the case $i=1$, $j=-1$. Taking equation (\ref{eq:clave2}) for these values, we have $$a_0(n_1a_1a_{-1}'-n_{-1}a_{-1}a_1')=a_1a_{-1}(n_{-1}+n_1).$$ Now, $n_{-1}+n_1\neq 0$, because $n_{-1},n_1\geq 0$ and $n_{-1}=n_1=0$ would imply $0=-\frac12[\rho(L_{-1}),\rho(L_1)]=\rho(L_0)$, a contradiction since $\rho$ is injective. Hence $n_1a_1a_{-1}'-n_{-1}a_{-1}a_1'\neq 0$ and the proof is finished invoking Corollary \ref{cor:fundamental}. 

\end{proof}

\begin{prop}\label{p:aih}
If $n_0= 1$, then for every $i\in\sop(\g)$ there exists $h(z)\in V$ with $h'(z)\neq 0$ such that:
	\[
	a_i= \mu_i h(z)^{i+n_i} (-h'(z))^{-n_i} 
	\, . \]
\end{prop}

\begin{proof}
%
%
Let us set $h(z):= \frac{a_1}{a_0^{n_1}} \in V$.  After simplifiying equation (\ref{eq:clave2})  for $i=1$, $j=0$, we get 
	\[
	a_1=n_1 a_1{a_0'}- a_0{a'_1}.
	\]
Computing the derivative of $h(z)$  and using the previous relation we have $$h'(z)=\frac{1}{a_0^{n_1+1}}(a_0 a_1'-n_1a_1 a_0')=-\frac{a_1}{a_0^{n_1+1}}=-\frac{h(z)}{a_0}.$$  This shows that  $h'(z)\neq 0$ and  $a_0=-\frac{h(z)}{h'(z)}$, $a_1=\left( -\frac{h(z)}{h'(z)}\right)^{n_1} h(z)$. Plugging these expressions into the formula for  $a_i$ given in Proposition~\ref{p:a_ia_1a_0}, the conclusion follows. 
\end{proof}

\begin{prop}[Orders]\label{p:c-1es01}
Let $n_0=\ord(L_0)= 1$, then:
	\begin{enumerate}
		\item for ${\mathfrak{g}}=\sl(2)$, the triple $\{n_{-1},n_0,n_1\}$ is equal to one of the following cases: $\{0,1,2\}$, $\{1,1,1\}$ or $\{2,1,0\}$. Therefore, one has $n_i+n_j-1=n_{i+j}$ for every $-1\leq i,j\leq 1$ with $i\neq j$.
		\item for ${\mathfrak{g}}=\Witt_>$, the general term of the sequence $\{n_{-1},n_0,n_1,\ldots\}$ is given either by $n_i=1$ for all $i\geq -1$, or by $n_i=i+1$ for all $i\geq -1$. Therefore, one has $n_i+n_j-1=n_{i+j}$ for every $i,j\geq -1$ with $i\neq j$.
		\item for ${\mathfrak{g}}=\Witt_<$, the general term of the sequence $\{\ldots,n_{-1},n_0,n_1\}$ is given either by $n_i=1$ for all $i\leq 1$, or by $n_i=-i+1$ for all $i\leq 1$. Therefore, one has $n_i+n_j-1=n_{i+j}$ for every $i,j\leq 1$ with $i\neq j$.
		\item for ${\mathfrak{g}}=\Witt$, the sequence $\{\ldots,n_{-1},n_0,n_1,\ldots \}$ is constant and  $n_i=1$ for all $i\in\Z$. Therefore, one has $n_i+n_j-1=n_{i+j}$ for every $i,j\in\Z$ with $i\neq j$.
	\end{enumerate} 
\end{prop}

\begin{proof}
(1) Proposition~\ref{p:n_i=} states that $n_{-1}+n_1-1=n_0$. Since $n_0=1$ and $n_{-1}, n_1\geq 0$, the result follows. 

(2) In this situation, we claim that $n_i\geq 1$ for all $i\geq 0$. Indeed, if  there exists $i\geq 0$ with $n_i=0$. Then, $i\geq 1$ and  \eqref{e:ci>1}  imply that $n_j-1\geq n_{i+j}$ for all $j\geq 0$ with $j\neq i$. Repeating this argument for $i+j, 2i+j,\ldots$, we obtain  $n_j > n_{i+j} > n_{2i+j} > \ldots$; however this is not possible since $n_k\geq 0$ for all $k$.

Hence, the relations $n_{-1}+n_1-1=n_0$ (Proposition~\ref{p:n_i=}), $n_0=1$ and $n_i\geq 1$ for all $i\geq 0$ imply that  there are two cases, either $n_{-1}=0$ or $n_{-1}=1$.

Considering equation (\ref{eq:clave2}) for $j=-1$ we get $$a_0(n_ia_ia_{-1}'-n_{-1}a_{-1}a_i')=a_ia_{-1}(in_{-1}+n_i).$$ Since we have seen that $n_i\geq 1$, it follows that $in_{-1}+n_i\neq 0$ for every $i\geq 0$. Therefore $n_ia_ia_{-1}'-n_{-1}a_{-1}a_i'\neq 0$ and by Corollary \ref{cor:fundamental} we have:
	\[
	n_{i} + n_{-1} -1 \, =\, n_{i-1}\qquad \forall i\geq 0\, . 
	\]
A simple recurrence procedure shows that:
	\[
	n_i \,= \, n_0 + i (1-n_{-1}) \qquad \forall i\geq 0\, . 
	\]
Recalling that $n_0=1$ and that $n_{-1}$ is either equal to $0$ or $1$, the statement is proven. 

(3) We consider  the Chevalley involution \eqref{eq:Chevalley}. If $\rho_<$ is a representation of $\Witt_<$ then $\rho_>:=\rho_<\circ \Theta$ is a representation of $\Witt_>$. The claim follows immediately from part (2) since for every $i\leq 1$ one has $\ord(\rho_<(L_i))=\ord(\rho_>(L_{-i}))$.

(4) Since $\Witt_>$ and $\Witt_>$ are Lie subalgebras of $\Witt$, the claim follows immediately from (2) and (3).
%
%
%
\end{proof}

\begin{thm}[Symbols]\label{t:symbolorder} 
If $n_0=\ord(L_0)= 1$, then there exists $h(z)\in V$ with $h'(z)\neq 0$ such that for every $i\in\sop(\g)$ one has:
\[ 
a_i = h(z)^{i+n_i}(-h'(z))^{-n_i}.
\]
\end{thm}

\begin{proof} It suffices to show that  $\mu_i=1$ for all $i\geq -1$ in Proposition~\ref{p:aih}. Thanks to   Proposition \ref{p:c-1es01} one has $n_i+n_j-1=n_{i+j}$ and by Corollary \ref{cor:fundamental} this implies the equality $$n_i a_i a'_j- n_j a'_i a_j=(i-j)a_{i+j}.$$ Substituting it into equation (\ref{eq:clave2}) yields the identity $$(i-j)a_0 a_{i+j} =(in_j-jn_i)a_ia_j.$$
Replacing $a_i$, $a_j$ by their expressions obtained in Proposition \ref{p:aih} and taking into account that $a_0=-\frac{h(z)}{h'(z)}$, $n_{i+j}=n_i+n_j-1$, we get  $$(i-j) \mu_{i+j} =(in_j-jn_i)\mu_i\mu_j.$$  

Substituting $n_i=1$ or $n_i=i+1$ or $n_i=-i
+1$ (see Proposition~\ref{p:c-1es01}), it follows the relation $\mu_{i+j}=\mu_i \mu_j$. Recalling that $\mu_1=1$, we have $\mu_{i+1}=\mu_i$ and since $1\in\sop(\g)$, this immediately implies:
	\[
	\mu_i=1\, , \qquad \forall i\in\sop(\g)\, .
	\]
\end{proof}

\begin{prop}\label{p:comm-constant}
Let $n_0=\ord(L_0)= 1$. Let $D\in\Diff(V)$ be such that $[D,\rho(L_i)]=0$ for all $L_i\in{\mathfrak{g}}$, then $D$ is constant. 
\end{prop}

\begin{proof}
Let $D$ be a $n$-th order differential operator with symbol $d$. We have to show that $n=0$ and that $d\in \C$.  Assume that $n>0$. Considering the relation $[D,\rho(L_i)]=0$ and arguing as in the proof of Proposition~\ref{p:fundamental}, we get the identity  $d^{n_i}=\mu_i a_i^n$ with $\mu_i\in \C^*$.  Hence, for $i=0$ we have $d=\mu_0 a_0^n$ and substituting this in the previous identity we obtain  $\mu_0^{n_i}a_0^{nn_i}=\mu_i a_i^n$. Therefore $\left(\frac{a_0^{n_i}}{a_i}\right)^n\in\C^*$ and since $n>0$ it follows that $\frac{a_0^{n_i}}{a_i}\in\C^*$. However, thanks to Theorem \ref{t:symbolorder} one has $\frac{a_0^{n_i}}{a_i}=h(z)^{-i}$. Whence $h(z)\in\C^*$ and $h'(z)=0$ that contradicts Theorem \ref{t:symbolorder}. Thus, $n=0$
%
and $D\in\Diff^0(V)= V$. The relation $0=[D,\rho(L_0)]= -\frac{h(z)}{h'(z)}\partial (D)$ yields the result. 
\end{proof}
\section{Case ${\mathfrak{g}}=\sl(2)$}\label{sec:expre}

We say that an $\sl(2)$-representation $\rho$ is a Casimir representation if the Casimir operator $\mathcal C$  acts by a constant $\rho({\mathcal C})=(2\mu+1)^2$ where $\mu$, called the semi-level of $\rho$, is the unique complex number in the set $$\C_\mathrm{sl}:=\{z\in \C\colon \frak{Re}(\mu)= -\frac12, \Im(\mu)\geq 0\} \coprod \{z\in \C\colon \frak{Re}(\mu)> -\frac12\}.$$

\begin{thm}\label{t:Casimir} Let $V=\C(z), \Cz$. Let us consider the set:
	\[
	{\mathcal S}\,:=\, \big\{ \rho\in \Hom_{\text{Lie-alg}}(\sl(2), \Diff(V)) \text{ s.t. $\rho(L_0)$ is of first order} \big\} 
	\]
Every representation $\rho\in\mathcal S$  is a Casimir representation. 
\end{thm}

\begin{proof} The representation of the Casimir operator $\rho(\mathcal C)$ commutes with every $\rho(L_i)$ for $-1\leq i\leq 1$. Therefore, thanks to Proposition \ref{p:comm-constant} it follows that $\rho(\mathcal C)$ is a constant.
\end{proof}

The main result of this section is the following Theorem. 

\begin{thm}\label{t:repr-sl2}
	Let $V=\C(z), \Cz$. Let us consider the set:
	\[
	{\mathcal S}\,:=\, \big\{ \rho\in \Hom_{\text{Lie-alg}}(\sl(2), \Diff(V)) \text{ s.t. $\rho(L_0)$ is of first order} \big\} 
	\]
	Then, it holds that ${\mathcal S}= {\mathcal S}_0 \sqcup {\mathcal S}_1 \sqcup {\mathcal S}_2$ where:
	\begin{itemize}
		\item  ${\mathcal S}_i$ consists of those maps such that $\ord(\rho(L_{-1}))=i$ for $i=0,1,2$. 
		\item ${\mathcal S}_0$ is parametrized by the set $\mathcal T_0$ of triples $ (h(z), b(z), c)$  such that  $h(z), b(z) \in V$, $c \in\C_\mathrm{sl}$, with $h'(z)\neq 0$ and $\rho\in\mathcal S_0$ associated to such a triple has semi-level $c$.  The correspondence is  given by:
			\begin{equation}\label{e:rhoexp}
			\begin{split}
			\rho(L_{-1}) \,& =\, h(z)^{-1} \, ,
			\\
			\rho(L_{0}) \,&=\, -\frac{h(z)}{h'(z)} \partial + b(z) \, ,
			\\
			\rho(L_1) \,&=\, 	h(z) \left(\rho(L_0)^2 - \rho(L_0) -c(c+1)\right)\, .
			\end{split}
			\end{equation}
		\item  ${\mathcal S}_1$ is parametrized by  the set $\mathcal T_1$ of triples $(h(z),b(z),c)$ such that  $h(z),b(z)\in V$,  $c \in\C$, with $h'(z)\neq 0$ and $\rho\in\mathcal S_1$ associated to such a triple has semi-level $c$ or $-c-1$.  The correspondence is given by:
		\begin{equation}\label{e:Lihcbsl2}
		\rho(L_i) =
		\frac{-h(z)^{i+1}}{h'(z)} \partial  +
		(b(z)+i \cdot c) \cdot h(z)^i \, .
		\end{equation}
		\item ${\mathcal S}_2$ is parametrized by the set $\mathcal T_2$ of triples $ (h(z), b(z), c)$  such that  $h(z), b(z) \in V$, $c \in\C_\mathrm{sl}$, with $h'(z)\neq 0$ and $\rho\in\mathcal S_0$ associated to such a triple has semi-level $c$.  The correspondence is  given by:
			\[
			\begin{split}
			\rho(L_{-1}) \,& =\, h(z)^{-1} \left(\rho(L_0)^2 + \rho(L_0) -c(c+1)\right)\, ,
			\\
			\rho(L_{0}) \,&=\, -\frac{h(z)}{h'(z)} \partial + b(z) \, ,
			\\
			\rho(L_1) \,&=\, 	h(z)\, .
			\end{split}
			\]
			${\mathcal S}_2$ is in bijection with ${\mathcal S}_0$ by mapping $\rho\in\mathcal S_0$ to $\rho^\Theta=\rho\circ\Theta\in\mathcal S_2$, where $\Theta$ is the Chevalley involution  \eqref{eq:Chevalley}.
	\end{itemize}
\end{thm}

\begin{proof}
	The fact that ${\mathcal{S}}$ is the disjoint union of ${\mathcal{S}}_i$   follows easily from the first item of Proposition~\ref{p:c-1es01}. It is also straightforward to see that the Chevalley involution induces a bijection  ${\mathcal S}_0 \simeq {\mathcal S}_2$. 
	
\textsl{The case of ${\mathcal S}_0$}. Given a triple $ (h(z), b(z), c)$   as in the statement, a straightforward computation shows that $\rho$ defined as in~\eqref{e:rhoexp} is a Casimir representation of semi-level $c$. Conversely, let $\rho\in {\mathcal S}_0$ be given, and let us compute the associated triple. Bearing in mind Proposition~\ref{p:c-1es01} and Theorem~\ref{t:symbolorder}, the expression of  $\rho(L_{-1})$ shows that $h(z)\in V$. Let $b(z)\in V$ be defined by  $b(z):=\rho(L_0)+\frac{h(z)}{h'(z)}\partial\in \Diff^0(V)=V$. 

Thus, it remains to validate the expression for $\rho(L_1)$. Thanks to Theorem  \ref{t:Casimir} we know that $\rho$ is a Casimir representation of $\sl(2)$. Let $\mu\in\C_{\mathrm{sl}}$ be the semi-level of $\rho$.  Bearing in mind that $\rho(L_{-1})$ is invertible and that  $\rho({\mathcal{C}})=(2\mu+1)^2$ is a constant, from identity (\ref{eq:Casimir1}) we obtain:
	\[
	\rho(L_1)\,=\, \rho(L_{-1})^{-1}\left( \rho(L_0)^2 -\rho(L_0)-\mu(\mu+1) \right).
	\]

\textsl{The case of ${\mathcal S}_1$}. 	Given a triple $(h(z),b(z),c))$  as in the statement, a straightforward computation shows that $\rho$ given by~\eqref{e:Lihcbsl2} defines a Casimir representation such that $\rho(\mathcal C)=(2c+1)^2$, hence its semi-level is either $c$ or $-c-1$.
	Let us now determine the triple associated to a representation $\rho\in {\mathcal S}_1$. Recalling Theorem~\ref{t:symbolorder}, we may write:
	\[
	\rho(L_i)\,=\, h(z)^i( \rho(L_0) + c_i(z)) \qquad i=-1,1 
	\]
	for certain $h(z), c_{-1}(z), c_1(z)\in V$ with $h'(z)\neq 0$. Expanding the identity $[\rho(L_{0}), \rho(L_{-1})]=\rho(L_{-1})$, one gets:
	\[
	- \frac1{h'(z)}c_{-1}'(z) \, =\, 0
	\]
	and, thus, $c_{-1}(z)$ is a constant. Proceeding similarly with $[\rho(L_0),\rho(L_1)]$, one gets that $c_1(z)$ is constant too. Finally, the identity $[\rho(L_{1}),\rho(L_{-1})]=2\rho(L_{0})$ implies that $c_{1}(z)=- c_{-1}(z)$. Denoting this constant by $c$ and defining:
	\[
	b(z)\,:=\, \rho(L_0)+\frac{h(z)}{h'(z)}\partial
	\, \in \, \Diff^0(V) = V
	\]
	we get the expressions (\ref{e:Lihcbsl2}).

Recall now that  $\rho$ is a  Casimir representation and Theorem   \ref{t:Casimir}. Let $\mu\in\C_{\mathrm{sl}}$ be the semi-level of $\rho$. By means of (\ref{eq:Casimir2})  and the identities $\rho(L_{-1})=h(z)^{-1}(\rho(L_0)-c)$, $\rho(L_{-1})=h(z)(\rho(L_0)+c)$, we get the equality 
	\begin{equation}\label{eq:id-Casimir}
	\left( \rho(L_0)+\frac{1}{2}\right)^2-\left(\mu+\frac{1}{2}\right)^2=
	h(z)(\rho(L_0)+c)h(z)^{-1}(\rho(L_0)-c)\end{equation}
Since   $\rho(L_0)+\frac{h(z)}{h'(z)}\partial=b(z)$, one has $[h(z),\rho(L_0)]=\frac{h(z)}{h'(z)}\partial h(z)=h(z)$ and thus $h(z)\circ \rho(L_0)=(\rho(L_0)+1)\circ h(z)$.  This implies\begin{align*}h(z)(\rho(L_0)+c)h(z)^{-1}(\rho(L_0)-c)&=(\rho(L_0)+c+1)(\rho(L_0)-c)\end{align*} and substituting it into (\ref{eq:id-Casimir}) we get $$\left( \rho(L_0)+\frac{1}{2}\right)^2-\left(\mu+\frac{1}{2}\right)^2=\left( \rho(L_0)+\frac{1}{2}\right)^2-\left(c+\frac{1}{2}\right)^2.$$ Whence it holds that either $\mu=c$ or $\mu=-c-1$ and  $\rho(\mathcal C)=(2c+1)^2$.

\textsl{The case of ${\mathcal S}_2$}. The proof is completely analogous to the one given for the case $\mathcal S_0$.
	
\end{proof}

\begin{rem}\label{r:Chevalley} 
Let us make explicit the action of the Chevalley involution on the set ${\mathcal S}_1$. Recall that, for a representation $\rho$, one defines another representation $\rho^{\Theta}(L_i):=(-1)^{i+1} \rho(L_{-i})$. Hence, if $\rho\in {\mathcal S}_1$ is associated to a triple $(h(z),b(z),c)\in\mathcal T_1$, then $\rho^{\Theta}$ is associated to another triple $( h^\Theta(z), b^\Theta(z),c^\Theta)\in\mathcal T_1$. Writing down the relations $\rho^{\Theta}(L_i)=(-1)^{i+1} \rho(L_{-i})$, one finds:
	\[
	h^\Theta(z) \, =\, - \frac1{h(z)} \quad , \quad
         b^\Theta(z)\, =\, -b(z)  \quad , \quad
         c^\Theta \, =\, c\, .
	\] In a similar way, the bijection of $\mathcal S_0$ with $\mathcal S_2$ induced by the Chevalley involution gives a bijection of the space of triples $\mathcal T=\mathcal T_0=\mathcal T_2$ such that for $(h(z),b(z),c)\in \mathcal T$ one has: \[
	h^\Theta(z) \, =\,  \frac1{h(z)} \quad , \quad
         b^\Theta(z)\, =\, -b(z)  \quad , \quad
         c^\Theta \, =\, c\, .
	\] 
\end{rem}

\subsection{Enveloping Algebra}\label{subsec:env}

\begin{thm}\label{t:Un-Env-Al-sl1}
	Let $\rho\in {\mathcal{S}}$ be the representation associated to a triple $(h(z),b(z),c)$. If  ${\mathcal U}(\sl(2))$ is the universal enveloping algebra of $\sl(2)$ and ${\mathcal C}$ is the Casimir operator, then  $\rho$ induces an injection:
		\[
		{\mathcal U}(\sl(2)) / \langle {\mathcal C} - (2c+1)^2 \rangle \, \hookrightarrow\, \Diff^1(\Cz)
		\, . \] 
\end{thm}

\begin{proof}
Let us prove the case $\rho\in {\mathcal{S}}_1$, the cases $\rho\in {\mathcal{S}}_0, {\mathcal{S}}_2$ can be proved similarly. 

One has to check that $( {\mathcal C} - (2c+1)^2)$ generates the kernel of the induced map ${\mathcal U}(\sl(2))\to \Diff^1(\Cz)$.  The Poincaré–Birkhoff–Witt Theorem implies that $\{ L_1^{\alpha}L_0^{\beta}L_{-1}^{\gamma}\,\vert\, \alpha, \beta, \gamma\geq 0\}$ is a basis of ${\mathcal U}(\sl(2))$ as a $\C$-vector space. Let us denote $L:=\rho(L_0)$. Using the identities $[h(z)^j, L]= j h(z)^j $ and $\rho(L_i):= h(z)^i(L+i c)$, one proves the following relations:	
	\begin{align*}
	&\rho(L_{-1}^{\gamma})\, =\, h(z)^{-1}(L-c) \, \overset{\gamma}\cdots\, h(z)^{-1}(L-c) 
		\, =\, h(z)^{-\gamma} P(L-c,\gamma),
	\\
	&\rho(L_{1}^{\alpha})\, =\, h(z)(L+c) \, \overset{\alpha}\cdots\, h(z)(L+c)
		\, =\, h(z)^{\alpha} P(L+c - \alpha+1, \alpha),
	\end{align*}
where $P$ is the Pochhammer symbol (see~\eqref{e:Poch}). Accordingly,
	\[
	\begin{aligned}
	&\rho(L_1^{\alpha}L_0^{\beta}L_{-1}^{\gamma})\,  =\, \\&=
	h(z)^{\alpha-\gamma} P(L+c- \alpha+\gamma+1, \alpha) (L+\gamma)^{\beta}   P(L-c,\gamma)
	\,=\\ & =\,  h(z)^{\alpha-\gamma} (L^{\alpha+\beta+\gamma}+\cdots)
	\end{aligned}\]
and note that the exponent of $h(z)$ is equal or smaller than that of $L$. 

If there is a linear combination of monomials $L_1^{\alpha}L_0^{\beta}L_{-1}^{\gamma}$ lying in the kernel, then it implies that there are two triples $\alpha, \beta, \gamma$ and $\alpha', \beta', \gamma'$ such that $\alpha-\gamma=\alpha'-\gamma'$ and $\alpha+\beta+\gamma= \alpha'+\beta'+\gamma'$. Recalling expression (\ref{eq:Casimir1}) for the Casimir operator and that it acts by a constant, it follows that:
	\[
	\rho(L_0)^2 \,=\, \rho(L_0) + \rho(L_{-1})\rho(L_1)+ \frac14(\rho(\mathcal C)-1)
	\]
and, accordingly, we can assume that $0\leq \beta, \beta'\leq 1$. 

It is enough to show that given two integer numbers $a,b$ with $b\geq 0$ and $b\geq a$, there is a unique solution of:
	\[
	\left\{
	\begin{aligned} & a\, =\, \alpha-\gamma \\  & b= \alpha+\beta+\gamma \end{aligned}
	\right.
	\]
where $\alpha,\beta,\gamma$ are non negative integer numbers and $\beta=0,1$. Indeed, the unique solution is given as follows. If $a,b$ have the same parity, then $\alpha=\frac12(a+b)$, $\beta=0$, $\gamma=\frac12(-a+b)$. Similarly, if $a,b$ have different parities, then  	$\alpha=\frac12(a+b-1)$, $\beta=1$, $\gamma=\frac12(-a+b-1)$. 
\end{proof}

\subsection{Conjugated Representations} 

In this subsection we consider the action of the  group of semilinear transformations on the set of representations ${\mathcal S}$. For the definition of this group and its action on ${\mathcal S}$ we address the reader to \cite[\S2.2]{Plaza-Opers}. More precisely, let  $\SGL_{\C((z))}(V)$ denote the group of semilinear transformations of $V:=\Cz$ as a $\Cz$-vector space. This group consists of those $\C$-linear automorphisms $\gamma:V\to V$ such that there exists a $\C$-algebra automorphism $\Phi$ of $\Cz$ satisfying:	
\[
\gamma(f(z)\cdot v)\,=\, \Phi(f(z)) \cdot \gamma(v)\qquad \forall f(z)\in \Cz, v\in V \, . 
\]

There is a bijection between the group  $\operatorname{Aut}_{\text{$\C$-alg}}(\Cz)$ and the  invertible series $\phi(z)\in \Czz^*$. The association between an automorphism $\Phi$ and a series $\phi(z)$ is given by the relation $\Phi(f(z))=f(\phi(z))$. The group of automorphisms also acts on the group of homotheties by conjugation; i.e. given an homothety $H_{s(z)}$ of ratio $s(z)\in  \Cz^*$, and an automorphism $\Phi$, one has that $\Phi(H_{s(z)}) = \Phi\circ H_{s(z)}\circ \Phi^{-1}=H_{s(\phi(z))}$ is the homothety of ratio  $s(\phi(z))$.  Accordingly,  $\SGL_{\C((z))}(V)$ can be identified with  the following semidirect product:
\[
\SGL_{\C((z))}(V) \, =\, \operatorname{Aut}_{\text{$\C$-alg}}(\Cz)\ltimes \Cz^*
\, . \]
It is worth pointing out that the Lie algebra of $\SGL_{\C((z))}(V)$ consists of first-order
differential operators on $V$. 

This group acts on  ${\mathcal S}$ by conjugation:
\[
(\gamma,\rho) \,\mapsto \, \rho^\gamma \qquad\text{  where }\rho^\gamma(L_k):= \gamma\circ\rho(L_k)\circ\gamma^{-1}\quad \forall k
\]
for $\gamma\in \SGL_{\C((z))}(V) $ and $\rho\in {\mathcal S}$. Let us describe explicitly the action  on  ${\mathcal S}$. 	
An automorphism  $\Phi\in \operatorname{Aut}_{\text{$\C$-alg}}(\Cz)$ acts by:
	\[
	(\Phi,\rho) \,\mapsto \, \rho^\Phi \qquad\text{  where }\rho^\Phi(L_k):= \Phi\circ\rho(L_k)\circ\Phi^{-1}\quad \forall k\, .
	\]
Hence, the transformation $\Phi$ transforms triples as follows:\todo{¿poner el cálculo?}
		\begin{equation}\label{e:Aut-act}
	 \left(\Phi , (h(z),b(z),c)\right) \,\mapsto \, (h(\phi^{-1}(z)),b(\phi^{-1}(z)),c)\, .
	 \end{equation}
Now, we write down the action of $\C((z))^*$ on  ${\mathcal S}$; which is given by:
	\[
	\left(s(z), \rho\right)\,\mapsto \, \left( s(z) \circ \rho \circ   s(z)^{-1}\right)
	\]
so that, in terms of triples, it holds that:
	\begin{equation}\label{e:homot-act}
	\left(s(z), (h(z),b(z),c) \right)\,\mapsto \, \left(h(z),b(z)+\frac{h(z)}{h'(z)}\frac{s'(z)}{s(z)},c\right)\, .
	\end{equation}
One now  checks that the first action intertwines the second one; that is:
	\[
	(\Phi\circ H_{s(z)})(\rho)= \Phi\circ H_{s(z)} \circ \rho \circ H_{s(z)^{-1}} \circ \Phi^{-1} =
	(H_{s(\phi(z))} \circ \Phi) (\rho)
	\]
 and, accordingly, the group $\SGL(\C((z)))$ acts on  ${\mathcal S}$.



\subsection{Relation with differential operators on the line}\label{subsec:sl2-line}

The above results  will allow us to classify representations $\rho:\sl(2)\to  \Diff(\Czz) $ such that $\rho(L_0)$ is a first order differential operator up to equivalence. Note that for all $k\geq 0$, the inclusion $\Czz \subset \Cz$ implies easily that:
\begin{equation}\label{e:DiffC[[z]]}
\Diff^k(\Czz) \, =\, \big\{ \sum_{j=0}^k \xi_j(z)\partial^j  \in \Diff^k(\Cz) \text{ s.t. } \xi_j(z)\in \Czz\big\}\, .
\end{equation}
Hence, given $\rho$ as above, it follows that $\rho\in {\mathcal S}$ and, thus, we distinguish the three cases $\rho\in {\mathcal S}_0$, $\rho\in {\mathcal S}_1$ or $\rho\in {\mathcal S}_2$. We proceed now to describe them.

\begin{thm}\label{t:sl2-C[[z]]}
Let $\rho \in \Hom_{\text{Lie-alg}}(\sl(2), \Diff^1(\Czz)) $. Then, there exist $\gamma \in \SGL_{\Czz}(\Czz)$ and  $c\in\C$ such that:
	 \[
	 \operatorname{Im}(\rho^\gamma)
	 \, =\, 
	 \langle \partial , z \partial - c , z^2\partial -2 c z \rangle\, .
	 \]
\end{thm}

\begin{proof}
Given $\rho$ as in the statement and bearing in mind~\eqref{e:DiffC[[z]]}, it follows that $\rho\in {\mathcal S}_1$ and, thus, let $(h(z),b(z),c)\in\mathcal T_1$ be its associated triple. 

From the very definition of the Chevalley involution $\Theta$, it is clear that $\rho$ and $\rho^\Theta$ (see Remark~\ref{r:Chevalley}) have the same image. Bearing in mind how  $\Theta$ acts on triples, one may assume that $\nu(h(z))\geq 0$, where $\nu$ denotes the valuation given by $z$. 

One can make a second assumption. Namely, given $\rho$ associated to a triple $(h(z),b(z),c)\in\mathcal T_1$ and a constant $a\in \C$, let us consider the representation $\rho^a$ associated to $(h(z)+a, (h(z)+a)(b(z)-c)h(z)^{-1}+c, c)$, notice that this triple is correctly defined  since $\rho(L_{-1})\in\Diff^1(\C[[z]])$ and by  \ref{e:DiffC[[z]]} this implies  $(b(z)-c)h(z)^{-1}\in\C[[z]])$. One checks that:	
	\[ 
	\begin{aligned}
	&\rho^a(L_{-1}) \, =\, \rho(L_{-1}) \\
	&\rho^a(L_0) \,=\, \rho(L_0) + a \rho(L_{-1})  \\
	&\rho^a(L_1) \,=\, \rho(L_1) + 2a\rho(L_0) + a^2 \rho(L_{-1})  \, .
	\end{aligned}\]
and, accordingly, $\operatorname{Im}\rho^a=\operatorname{Im}\rho$. So, we can assume $h(0)=0$ or, what is tantamount, $\nu(h(z))\geq 1$. 

From \eqref{e:DiffC[[z]]}, one has  $\frac{-1}{h'(z)}\in \Czz$ and thus, $\nu(h'(z))\leq 0$. 

Summing up, it follows that $\nu(h(z))= 1$ and $\nu(h'(z))=0$. Consider the $\C$-algebra automorphism of $\Czz$ given by $\Phi(f(z)) := f(h(z))$. Thanks to \eqref{e:Aut-act}, we can replace $\rho$ by $\rho^\Phi$. That is, we may assume that $\rho$ is associated to the triple $(z,b(z),c)\in\mathcal T_1$. Again, \eqref{e:DiffC[[z]]} implies that  $h(z)^{-1}(b(z)-c)\in \C[[z]]$ since $\rho(L_{-1})\in\Diff^1(\C[[z]])$.  Therefore,  there exists $s(z)\in \Czz$  such that $c-b(z)=\frac{h(z)}{h'(z)}\frac{s'(z)}{s(z)}$. Hence, conjugating by the homothety of ratio $s(z)$, and recalling \eqref{e:homot-act}, one concludes that the triple we started with can be assumed to be  $(z, c, c)\in\mathcal T_1$. Denote by  $\bar \rho$ its associated representation. Writing down the operators $\bar\rho(L_{-1})$, $\bar \rho(L_0)$ and $\bar\rho(L_{-1})$, one obtains the result. 
\end{proof}

\begin{rem}
	This Theorem, which studies the case $\rho\in {\mathcal S}_1$, can be thought of as an algebraic analogue of the classical result that  every representation of $\sl(2)$ by first order differential operators of the ring of real (or complex) analytic functions in the line is equivalent to $\langle \partial , z \partial + \lambda , z^2\partial +2 \lambda z \rangle $. 	See Miller \cite[\S8]{Miller} (see also \cite[Thm 1]{Kamran-Olver} for another proof and applications). Let us point out that the classification of finite dimensional Lie subalgebras on the module of derivations in one complex variable dates back to Lie itself (see \cite[vol. III]{Lie}). In order to be more precise, let us recall \cite[\S8.2]{Miller}, where the author studies the representations $\rho$  in terms of first order differential operators on ${\mathcal A}$, the space of complex  functions in ${\mathbb{C}}$ which are analytic in a neighborhood of $0\in{\mathbb{C}}$. For this goal, one fixes $\rho_0$, a realization of $\sl(2)$ by derivations of ${\mathcal A}$ and one introduces an equivalence relation in the set of those $\rho$ with the same $\rho_0$. The equivalence relation turns out to coincide with the   action of $\Cz^*$ defined in \eqref{e:homot-act}. Then it is shown (\cite[\S8.2]{Miller}) that the set of equivalence classes are given by the first cohomology group of $\sl(2)$ with values in ${\mathcal A}$ which is equal to ${\mathbb{C}}$, and  the representation $\langle \partial , z \partial + \lambda , z^2\partial +2 \lambda z \rangle $ is mapped to $\lambda$.
\end{rem}

\begin{rem}
The Lie algebra $\langle \partial , z \partial + \lambda , z^2\partial +2 \lambda z \rangle $ is a central object in the results of \cite{Turbiner-sl} when characterizing  differential equations admitting polynomial solutions. In fact, \cite[Lemma 1]{Turbiner-sl} can be deduced from what we have proved. Furthermore, due to our results for the enveloping algebra of $\sl(2)$ (see~\S\ref{subsec:env}) and for the  Witt algebra  (see~\S\ref{subsec:env-alg-w}), one could extend Turbiner's  characterization to Witt algebras. 
\end{rem}

\begin{thm}\label{t:sl2-line}
	Let $\rho \in \Hom_{\text{Lie-alg}}(\sl(2), \Diff(\Czz))$ be a representation such that $\rho\in {\mathcal{S}_0}$. Then $\rho$ is equivalent, under the action of $\SGL_{\Czz}(\Czz)$, to one of the following three types:
	\begin{enumerate}
		\item 	$\bar{\rho}(L_{-1})= z^2$, $\bar{\rho}(L_0)= \frac1{2}z \partial +\frac14$, $\bar{\rho}(L_{1})= \frac14\partial^2$;
		\item  $\bar{\rho}(L_{-1})= z$, $ \bar{\rho}(L_0)= z \partial +b$, $\bar{\rho}(L_{1})= z\partial^2 + 2b\partial$, for $b\in \C$;
		\item $\bar{\rho}(L_{-1})= {a+z}$, $\bar{\rho}(L_0)= (a+z) \partial$, $\bar{\rho}(L_{1})= (a+z) \partial^2 -c(c+1)(a+z)^{-1}$ for $a\in \C^*$ and $c\in\C_{\mathrm{sl}}$.		
	\end{enumerate}
\end{thm}
	
\begin{proof}
	Since $\rho\in{\mathcal{S}_0}$, let us consider its associated triple $(h(z),b(z),c)$. Recalling Theorem~\ref{t:symbolorder} and \eqref{e:DiffC[[z]]}, one gets:
		\[
		(2i+1)\nu(h(z))-(i+1) \nu(h'(z))\geq 0\quad \text{ for }i=-1,0,1\, ,
		\]
	 which yields the following three cases. 
	
	$(1)$. $\nu(h(z))=-2$. Hence, acting by an automorphism of $\Czz$, we may assume $h(z)=z^{-2}$. Since $\frac{h(z)}{h'(z)}=-\frac1{2}z$ and bearing in mind \eqref{e:homot-act} we may conjugate by an homothety $s(z)\in\Czz^*$ such that $b(z)+\frac{h(z)}{h'(z)}\frac{s'(z)}{s(z)} $ does  not depend on $z$. That is, we can  assume that $b(z)$ is a constant, say $b\in\C$. Using the expression \eqref{e:rhoexp} to compute the coefficients of $\partial$ in the operator $\rho(L_{1})$ and noting that they must lie in $ \Czz$, further constraints follow. Indeed, looking at the coefficient of $\partial$, one gets $b=\frac14$. Similarly, the fact that the free term of $\rho(L_1)$ belongs to $\Czz$ yields $b^2- b-c(c+1)=0$ and, thus, $c=-\frac1{4}\in\C_{\mathrm{sl}}$. 
	
	Summing up, the representation we started with is equivalent, under the action of $\SGL(\Czz)$, to the representation associated to the triple $(z^{-2},\frac14,-\frac1{4})\in\mathcal T_0$; that is,
		\[
		\bar{\rho}(L_{-1})= z^2, \qquad
		\bar{\rho}(L_0)= \frac1{2}z \partial +\frac14, \qquad
		\bar{\rho}(L_{1})= \frac1{4}\partial^2 \, .		
		\]
	
	$(2)$. $\nu(h(z))=-1$. Proceeding as above, acting by a suitable automorphism and a homothety, we may assume that $h(z)=z^{-1}$ and $b(z)$ is  constant, say $b\in \C$. Then, the coefficients of $\rho(L_1)$ lie in $\Czz$  if and only if $b^2- b-c(c+1)=0$. Hence, $\rho$ is equivalent to the representation associated to the triple $(z^{-1},b,c)\in\mathcal T_0$ where either $b=-c$ or $b=c+1$:
			\[
		\bar{\rho}(L_{-1})= z, \qquad
		\bar{\rho}(L_0)= z \partial +b, \qquad
		\bar{\rho}(L_{1})= z\partial^2 + 2b\,\partial \, .		
		\]
				
	$(3)$.  $\nu(h(z))=\nu(h'(z))=0$. Acting by a suitable automorphism and a homothety, we may assume that $h(z)=(a+ z)^{-1}$ for $a\in\C^*$ and that $b(z)=0$. There are no further constraints and, thus,   $\rho$ is equivalent to the representation associated to the triple $((a+z)^{-1},0,c)\in\mathcal T_0$:
	\[
	\bar{\rho}(L_{-1})= {a+z}, \,
	\bar{\rho}(L_0)= (a+z) \partial , \,
	\bar{\rho}(L_{1})= (a+z) \partial^2 -c(c+1)(a+z)^{-1} \, .		
	\]
\end{proof}	

Following the same ideas or just by using the Chevalley involution, one gets the following:

\begin{thm}\label{t:sl2-line-rho-2}
	Let $\rho \in \Hom_{\text{Lie-alg}}(\sl(2), \Diff(\Czz))$ be a representation such that $\rho\in {\mathcal{S}_2}$. Then $\rho$ is equivalent, under the action of $\SGL_{\Czz}(\Czz)$, to one of the following three types:
	\begin{enumerate}
		\item 	$\bar{\rho}(L_{-1})= \frac14\partial^2 $, $\bar{\rho}(L_0)=- \frac1{2}z \partial -\frac14$, $\bar{\rho}(L_{1})= z^2$;
		\item  $\bar{\rho}(L_{-1})= z\partial^2 - 2b\partial$, $ \bar{\rho}(L_0)= -z \partial +b$, $\bar{\rho}(L_{1})= z$, for $b\in \C$;
		\item $\bar{\rho}(L_{-1})= (a+z) \partial^2 -c(c+1)(a+z)^{-1} $, $\bar{\rho}(L_0)= -(a+z) \partial$, $\bar{\rho}(L_{1})= {a+z}$ for $a\in \C^*$ and $c\in\C_{\mathrm{sl}}$.		
	\end{enumerate}
\end{thm}

\section{The Cases of Witt and Virasoro algebras}\label{sec:Witt}

Let the Pochhammer symbol be defined by:
\begin{equation}\label{e:Poch}
P(f,n) \,:=\, 
\begin{cases} f (f+1)\cdot\ldots \cdot (f+n-1)  & \text{ for } n> 0 \\
1  & \text{ for } n=0 \\
\frac1{(f+n)(f+n+1)\dots (f-1)} & \text{ for } n\leq -1
\end{cases}
\end{equation}
and note that it makes sense for $f$ in any ring with unity for $n\geq 0$ and in any field of characteristic $0$ and $f\notin {\mathbb{Z}}$ for $n<0$. 

\begin{thm}\label{t:Witt>}
	Let ${\mathfrak g}=\Witt_{>}, \Witt_<, \Witt, \Vir$ and $V=\C(z), \Cz$. It holds that:
\begin{enumerate}
	\item[(a)] 	\[ 
	{\mathcal R}(\g)\,:=\, \left\{\begin{gathered}  \rho\in \Hom_{\text{Lie-alg}}({\mathfrak{g}}, \Diff(V)) \\  
	\text{ s.t. $\rho(L_0)$ is of first order} \end{gathered} \right\} 
	\, = \, {\mathcal R}_0(\g) \sqcup {\mathcal R}_1(\g) \sqcup {\mathcal R}_2(\g)
	\]
	where $ {\mathcal R}_i(\g)$  consists of maps $\rho$ such that $\rho(L_{-1})$ has order $i$.
	\item[(b)] The embedding $\sl(2)\hookrightarrow {\mathfrak g}$ induces a restriction map:
	\[
	\begin{aligned}
	r\colon {\mathcal{R}}(\g) \,& \longrightarrow\, {\mathcal{S}}
	\\ \rho \, & \longmapsto \rho\vert_{\sl(2)}
	\end{aligned}
	\]
	that yields  a bijection $r\colon  {\mathcal{R}}_1(\g) \to {\mathcal{S}}_1$ and maps  $r\colon  {\mathcal{R}}_0(\g) \to {\mathcal{S}}_0$, $r\colon  {\mathcal{R}}_2(\g) \to {\mathcal{S}}_2$ whose fibers have at most cardinality $2$.
\end{enumerate}

More precisely, one has 	
	\begin{enumerate}
	\item Given $\rho\in  {\mathcal{R}}_1(\g)$, let $(h(z),b(z),c)\in\mathcal T_1$ be the triple associated to $\rho\vert_{\sl(2)} \in {\mathcal{S}}_1$, then:
		\begin{equation}\label{e:Lihcb}
		\rho(L_i) =
		\frac{-h(z)^{i+1}}{h'(z)} \partial  +(b(z)+i \cdot c) \cdot h(z)^i
		\quad\forall L_i\in{\mathfrak{g}}
		\end{equation}
	and, in the case ${\mathfrak{g}}= \Vir$, the central element is mapped to $0$, $\rho(K)=0$.
	\item If $\rho\in  {\mathcal{R}}_0( \Witt_>)$, and $(h(z),b(z),c)\in\mathcal T_0$ is the triple associated to $\rho\vert_{\sl(2)} \in {\mathcal{S}}_0$, then:
		\begin{equation}\label{e:Li012}
	\begin{split}
	\rho(L_{0}) \,&=\, -\frac{h(z)}{h'(z)} \partial + b(z) \, ,
	\\
	\rho(L_i) \,&=\, 	h(z)^i \left(\rho(L_0)  +i \lambda \right) 
	P(\rho(L_0) -\lambda-i, i)	 \, , \ \text{for}\  i\neq 0,
	\end{split}
		\end{equation}
	where   $\lambda$ is equal to either $c$ or $-c-1$ and $P$ is the Pochhammer symbol. Therefore,  the mapping $r\colon  {\mathcal{R}}_0(\Witt_>) \to {\mathcal{S}}_0$ is surjective and the fibers have cardinality $2$, except for $c=-\frac12$ where the cardinality is $1$.
	\item If $\rho\in  {\mathcal{R}}_2( \Witt_<)$, and $(h(z),b(z),c)\in\mathcal T_2$ is the triple associated to $\rho\vert_{\sl(2)} \in {\mathcal{S}}_2$, then:
		\begin{equation}\label{e:Li210}
	\begin{split}
	\rho(L_{0}) \,&=\, -\frac{h(z)}{h'(z)} \partial + b(z) \, ,
	\\
	\rho(L_i) \,&=\, 	h(z)^i \left(\rho(L_0)  +i \lambda \right) 
	P(\rho(L_0) +\lambda+1 ,-i)	 \, , \ \text{for}\  i\neq 0,
	\end{split}
		\end{equation}
	where   $\lambda$ is equal to either $c$ or $-c-1$ and $P$ is the Pochhammer symbol. Thus,  $r\colon  {\mathcal{R}}_0(\Witt_<) \to {\mathcal{S}}_2$ is surjective and its fibers have cardinality $2$, except for $c=-\frac12$ where the cardinality is $1$.  
	
	${\mathcal R}_2(\Witt_>)$ is in bijection with ${\mathcal R}_0(\Witt_<)$ by mapping $\rho\in\mathcal R_0$ to $\rho^\Theta=\rho\circ\Theta\in\mathcal R_2$, where $\Theta$ is the Chevalley involution  .

	\item The subsets  $ {\mathcal R}_0(\Witt_<) $, $ {\mathcal R}_2(\Witt_>) $ are empty.
	\item For  ${\mathfrak g}= \Witt, \Vir$, the subsets  $ {\mathcal R}_0(\g) $, $ {\mathcal R}_2(\g) $ are empty. 
\end{enumerate}

\end{thm}

\begin{proof}
	For the first part, let $\rho\in{\mathcal{R}}(\g)$ be given. Since $\ord(\rho(L_0))=1$, Proposition~\ref{p:c-1es01} implies that $\ord(\rho(L_{-1}))$ is either $0$ or $1$. Recall that these two cases correspond to $n_i=\ord(\rho(L_i))=i+1$ and  $n_i=\ord(\rho(L_i))=1$, respectively. Thus, applying Proposition~\ref{p:c-1es01} again, one checks easily that $r({\mathcal{R}}_i)(\g) \subseteq {\mathcal{S}}_i$ for $i=0,1,2$. 
	
	Now, let us study the map $r\colon  {\mathcal{R}}_1(\g) \to {\mathcal{S}}_1$. We begin with the case ${\mathfrak g}=\Witt_>$. Given $\rho\in  {\mathcal{S}}_1$, we know that $n_i=1$ for all $i=-1,0,1$. Then, the proof of \cite[Theorem~2.1]{Plaza-Opers} can be applied to both cases, $V=\C(z)$ and $V=\Cz$, and it yields that there exists a unique triple $(h(z),b(z),c)$ such that $\rho$ extends uniquely to a representation of $\Witt_>$ given by:
			\[
	\rho(L_i) =
	\frac{-h(z)^{i+1}}{h'(z)} \partial  +(b(z)+i \cdot c) \cdot h(z)^i
	\qquad \forall L_i\in 
	\Witt_>\, . \]
	Bearing in mind Theorem~\ref{t:repr-sl2}, one concludes that $r\colon  {\mathcal{R}}_1(\Witt_>) \to {\mathcal{S}}_1$ is bijective. 
	
	The case $\g=\Witt_<$. 	Note that  the Chevalley involution,  $\Theta$,  establishes an isomorphism between $\Witt_<$ and $\Witt_>$.  Since the result has already been proved for $\Witt_>$, given $\rho\in\mathcal R_1(\Witt_<)$ we may apply it to $\rho^\Theta:=\rho\circ\Theta \in\mathcal R_1(\Witt_>)$. Let $(h(z),b(z),c)$ be the triple associated to $\rho|_{\sl(2)}$. One has $(\rho|_{\sl(2)})^\Theta=(\rho^\Theta)|_{\sl(2)}$. Hence, thanks to Remark \ref{r:Chevalley}, $\rho^\Theta$ is given by formula (\ref{e:Lihcb}) for the triple  $( h^\Theta(z), b^\Theta(z),c^\Theta)$. A simple check shows that $\rho$ is given by the same formula (\ref{e:Lihcb}) for the triple  $( h(z), b(z),c)\in\mathcal T_1$. Therefore, $r\colon  {\mathcal{R}}_1(\Witt_<) \to {\mathcal{S}}_1$ is also bijective.

The case ${\mathfrak g}=\Witt$. Since $(\rho\vert_{\Witt_>})\vert_{\sl(2)}$ coincides with  $ (\rho\vert_{\Witt_<})\vert_{\sl(2)}$, it follows that $\rho$ is explicitly given by \eqref{e:Lihcb} for all $i\in {\mathbb Z}$. This immediately implies that $r\colon  {\mathcal{R}}_1(\Witt) \to {\mathcal{S}}_1$ is a bijection.
	
	Finally, let us deal with the case of $\Vir$. Notice that $\Witt_>$ is a  Lie subalgebra of $\mathrm{Vir}$ and we have proved that the restriction $\rho_>$ of $\rho$ to $\Witt_>$ is determined by a triple $\xi_>=(h_>(z), b_>(z),c_>)$. Analogously, $\Witt_<$ is a Lie subalgebra of $\mathrm{Vir}$ and we have seen that the restriction $\rho_<$ of $\rho$ to $\Witt_<$ is determined by another triple $\xi_<=(h_<(z), b_<(z),c_<)$. Since $\rho_>$ and $\rho_<$ do coincide on $\sl(2)=\Witt_>\cap \Witt_<$ we immediately get the equality $\xi_>=\xi_<$ and  therefore $\rho(L_i)$ must acquire the form \eqref{e:Lihcb} for all $i$. 
	Now it only remains to determine $\rho(K)$ where $K$ is the central element. Since $\rho$ is a map of Lie algebras  and ${\Vir}$ is defined by the cocycle~\eqref{eq:Vir-co}, it holds that $$\frac{1}{2}\rho(K)=[\rho(L_2),\rho(L_{-2})]-4\rho(L_0).$$
	Using the explicit expressions \eqref{e:Lihcb}, a simple computation shows now that the left hand side vanishes, therefore $\rho(K)=0$. This finishes the proof of (1).
	
It is clear that (3) follows from (2) and $(5)$ follows from (4) and this is a consequence of parts (2) and (3) of Proposition~\ref{p:c-1es01}.  
	
It remains to prove $(2)$. This follows from the  three Lemmas below.  
\end{proof}

It is worth noticing that, as a consequence of the above Theorem, any representation $\rho\in{\mathcal R}_1$ of $\Vir$ factorizes through a representation of $\Witt=\Vir/\langle K\rangle$ since $\rho(K)=0$.

The case $\rho: \Witt \hookrightarrow \Diff^1(\Cz)$, which corresponds to the condition $n_i=1$ for all $i\geq -1$,  has been exhaustively studied in \cite{Plaza-Opers}.

\begin{lem}\label{l:LiPoch}
Let $\mathfrak{g}=\Witt_>$, $V=\C(z), \Cz$ and $\rho\in {\mathcal{R}}_0$. Given $h(z), b(z)\in V$ with $h'(z)\neq 0$ and $c\in \C_{\mathrm{sl}}$, define $\rho(L_i)$ by the formulae~\eqref{e:Li012}.  

Then, $\rho$ is a homomorphism of Lie algebras and $\rho(L_i)$ coincides with the operators given in~\eqref{e:rhoexp} for $i=-1,0,1$.
\end{lem}

\begin{proof}
Note that $\rho(L_0)\circ h(z)^i = h(z)^i \circ (\rho(L_0)-i)$ and, thus, an analogous  commutation relation holds for polynomials in $\rho(L_0)$. 

For the sake of brevity, we denote $\rho(L_0)$ simply by $L$. Computing the Lie bracket explicitly in terms of Pochhammer symbols and using their properties, one gets:
	\begin{align*}
	[\rho(L_i) , & \rho(L_j) ]\, =\, 
	h(z)^{i+j}  \,\cdot 
	\\ 
	& \left( (L +i \lambda -j ) P(L -\lambda -i-j, i) (L +j\lambda ) P(L -\lambda-j, j)\right.
	\, - \\ &  -\,  
	\left.(L +j \lambda -i) P(L -\lambda -i-j, j) (L +i\lambda ) P(L -\lambda-i, i) \right)
	\,=\\ & = \, 
	(i-j) h(z)^{i+j} 	(L +(i+j) \lambda ) P(L -\lambda-(i+j), i+j)
	\end{align*}
and so the conclusion follows. 
\end{proof}

Observe that a morphism of Lie algebras from $\Witt_>$ to any Lie algebra is determined by its restriction to $\sl(2)$ and the image of $L_2$. Indeed, from \cite[\S2.3]{Plaza-Tejero} we know that given $\rho: \Witt_>\to {\mathfrak g}$, it holds that $\rho(L_{i+2})= \frac{(-1)^i}{i!} \ad(\rho(L_1))^i(\rho(L_2))$ for all $i>0$, where $\ad(M)(N):=[M,N]$. In the following two Lemmas, we use this fact.

\begin{lem}\label{l:constant}
Let $\mathfrak{g}=\Witt_>$, $V=\C(z), \Cz$ and $\rho,\rho'\in {\mathcal{R}}_0$.

If $r(\rho)= r(\rho')$, then, there exists $\alpha\in \C$ such that:
	\[
	\rho(L_2)= \rho'(L_2)+  \alpha h(z)^2
	\, .\]
\end{lem}

\begin{proof}
Consider $T:=\rho(L_2)- \rho'(L_2)$, which is a differential operator of order at most  $2$. Then, $T$ satisfies:
	\begin{align*}
	[ \rho(L_{-1}), T] \, & =\, 0
	\\
	[ \rho(L_{0}), T] \, & =\, -2T
	\end{align*}
Since $\rho(L_{-1})= h(z)^{-1}$, the first identity implies that $T$ is a differential operator of order $0$; that is, $T=t(z)\in V$. 

The second identity yields:
	\[
	- \frac{h(z)}{h'(z)} t'(z) \, =\, -2 t(z)
	\]
and, thus, $t(z)= \alpha h(z)^2$ for some $\alpha\in \C$. 
\end{proof}

\begin{lem}\label{l:21}
Let $\mathfrak{g}=\Witt_>$, $V=\C(z), \Cz$ and $\rho\in {\mathcal{R}}_0$.

Then, there is at most another $\rho'\in {\mathcal{R}}_0$ such that $r(\rho)= r(\rho')$.
\end{lem}

\begin{proof}
Having in mind Lemma~\ref{l:constant}, one should determine those $\alpha\in \C$ such that:
	\begin{align*}
	& \rho'(L_{i}):= \rho(L_i) \quad \text{for }i=-1,0,1
	\\
	& \rho'(L_{i+2}):= \frac{(-1)^i} {i!} \ad(\rho(L_1))^i 
	\left(\rho(L_2)+ \alpha h(z)^2\right) 
	\quad \text{for }i\geq 0
	\end{align*}
define a map of Lie algebras. 

The first condition that must be imposed is $[\rho'(L_2) , \rho'(L_3) ] 
= - \rho'(L_5)$. Using the definition of $\rho'$ and the fact that $\rho$ is a map of Lie algebras, this identity is equivalent to: 
	\begin{multline*}
	[  \alpha h(z)^2  , \rho(L_3) ] +
	[\rho(L_2) , - \ad(\rho(L_1))\left( \alpha h(z)^2 \right)] + \\ +
	[  \alpha h(z)^2 , - \ad(\rho(L_1))\left( \alpha h(z)^2\right)] 
	\,=\, -\frac16 \ad(\rho(L_1))^3 \left( \alpha h(z)^2\right)
	\end{multline*}
which is a degree $2$ equation in $\alpha$. A solution of this equation is $\alpha=0$ which corresponds to $\rho$. Hence the other solution, if it exists, yields the desired map $\rho'$. 
\end{proof}

\subsection{Differential Operators on the Line} Now, we study the case of representations as operators on $V= \Czz$.

\begin{thm}\label{t:Witt>-line}
	It holds that the set:
	\[ 
	\left\{\begin{gathered}  \rho\in \Hom_{\text{Lie-alg}}( \Witt_{>} , \Diff( \Czz )) \\  
	\text{ s.t. $\rho(L_0)$ is of first order} \end{gathered} \right\} 
	\]
consists of those $\rho$, that up to the action of $\SGL_{\Czz}( \Czz)$, are of the following types:	
	\begin{enumerate}
		\item 	
		$\ord(\rho(L_i))=1$ for all $i\geq -1$ and  $\rho(L_i)$ is given by~\eqref{e:Lihcb} for a triple  $(a+z,0,c)\in\mathcal T_1$; 
		\item $\ord(\rho(L_i))=i+1$ for all $i\geq -1$  and  $\rho(L_i)$ is given by~\eqref{e:Li012} for $\lambda\in\C$ and  a triple $(h(z),b(z),c)\in\mathcal T_0$
		as follows:
		 \begin{itemize}
		 	\item $(z^{-1},1,0)$ with $\lambda=-1$;
		 	\item $(z^{-1},\frac12,-\frac12)$ with $\lambda=-\frac12$;
		 	\item $(z^{-1},0,0)$ with $\lambda=0$;
			\item $(z^{-1},c+1,c)$ with $\lambda=c$, $c\in \C_{\mathrm{sl}}$;
			\item $(z^{-1},-c,c)$ with $\lambda=-c-1$, $c\in \C_{\mathrm{sl}}$;
			\item $((z+a)^{-1},0,c)$ with $\lambda\in\{c,-c-1\}$, $c\in \C_{\mathrm{sl}}$, $ a\in\C^*$;
		 \end{itemize}
	\end{enumerate}	
\end{thm}

\begin{proof}
Bearing in mind the inclusion $\Diff^k(\Czz)\hookrightarrow \Diff^k(\Cz)$, we may apply Theorem~\ref{t:Witt>} to $\rho$ and conclude that the set of the statement is identified with:
		\[ 
	\left\{\begin{gathered}  \rho\in \Hom_{\text{Lie-alg}}({\mathfrak{g}}, \Diff(\Cz)) \text{ s.t. $\rho(L_0)$ is of first order}
	\\  
	\text{and $\rho(L_i)(\Czz)\subseteq \Czz$ for $i=-1,0,1,2$} \end{gathered} \right\} 
	\]


$(1)$. If $\rho\in {\mathcal{R}_1}$ or, what is tantamount, $\ord(\rho(L_i))=1$ for all $i\geq -1$, equation ~\eqref{e:DiffC[[z]]} and Theorem~\ref{t:symbolorder} imply that $(i+1)\nu(h(z))-\nu(h'(z))\geq 0$ for all $i\geq -1$. The only possibility is $\nu(h(z))=\nu(h'(z))=0$. Therefore, acting by a suitable automorphism and a homothety, we may assume that $h(z)=a+z$ with $a\in\C^*$ and $b(z)=0$. 

$(2)$. If $\rho\in {\mathcal{R}_0}$ or, $\ord(\rho(L_i))=i+1$ for all $i\geq -1$, we may apply Theorem~\ref{t:sl2-line} to $\rho\vert_{\sl(2)}$ and obtain three cases. Therefore, it remains to impose the condition $\rho(L_2)(\Czz)\subseteq \Czz$. For this goal we recall the triples associated to $\rho\vert_{\sl(2)}$ by Theorem~\ref{t:sl2-line} and the explicit expression for $\rho(L_2)$ given by \eqref{e:Li012}: 
	\[ 
	\begin{aligned}
	\rho(L_2)=h(z)^{2}(\rho(L_0)+2\lambda)(\rho(L_0)-\lambda-2)(\rho(L_0)-\lambda-1)\, = \\ = \,
	h(z)^{2}\left( L^3 - 3L^2 +(2-3\lambda-3\lambda^2)L + 4\lambda +6 \lambda^2 +2\lambda^3\right)
	\end{aligned}
	\]  
with $L:=\rho(L_0)$ and $\lambda$ is equal to $c$ or $-c-1$.

In the first case, the triple was $(z^{-2}, \frac14,-\frac14)$.  Since $L=\frac12 z\partial +\frac14$, the highest order term in $\partial$ in $\rho(L_2)$  is $\frac18 h(z)^2z^3\partial^3=\frac18z^{-4} z^3\partial^3=\frac18 z^{-1}\partial^3$. 
So, there is none of these representations. 

For the second case, the associated triple  is $(z^{-1},b,c)$, with $b=-c$ or $b=c+1$.  Since $L=z\partial +b$, the desired condition is equivalent to saying that the free term $A(b,\lambda)$ and the coefficient $B(b,\lambda)$ of $z \partial$ in $\rho(L_2)$ vanish.  A straightforward computation shows 
	\begin{align*}A(b,\lambda)&=-3 b^2+b^3+4 \lambda+6 \lambda^2+2 \lambda^3+b\left(-3 \lambda^2-3 \lambda+2 \right),\\ 
	B(b,\lambda)&=3 (b^2- b- \lambda^2- \lambda).
\end{align*}

Now, there are four possibilities:

\begin{enumerate}
\item $b=-c$, $\lambda=c$.  One has: $$A(-c,c)=2 c \left(2 c^2+3 c+1\right),\quad B(-c,c)=0.$$ Hence $c\in\{-1,-\frac12,0\}$, but since $c\in\C_{\mathrm{sl}}$, the only possible values are $c=-\frac12$, $c=0$, that correspond, respectively,  to $(z^{-1},\frac12,-\frac12)$ with $\lambda=-\frac12$ and $(z^{-1},0,0)$ with $\lambda=0$.

\item $b=-c$, $\lambda=-c-1$. One has $$A(-c,-c-1)=0,\quad B(-c,-c-1)=0.$$ Hence there are no restrictions on $c\in \C_{\mathrm{sl}}$. This corresponds to the triple $(z^{-1},-c,c)$ with $\lambda=-c-1$ .

\item $b=c+1$, $\lambda=c$. One has $$A(c+1,c)=0,\quad B(c+1,c)=0.$$ Hence there are no restrictions on $c\in \C_{\mathrm{sl}}$. This corresponds to the triple $(z^{-1},c+1,c)$ with $\lambda=c$ .

\item $b=c+1$, $\lambda=-c-1$. One has $$A(c+1,-c-1)=-2 c \left(2 c^2+3 c+1\right),\quad B(c+1,-c-1)=0.$$ As in case 1), the only  possible values for $c\in \C_{\mathrm{sl}}$ are $-\frac12$ and $0$. They correspond, respectively,  to the triples $(z^{-1},\frac12,-\frac12)$ with $\lambda=-\frac12$ and $(z^{-1},1,0) $ with $\lambda=-1$.
\end{enumerate}

%

Finally, the last case given in Theorem~\ref{t:sl2-line} corresponds to the triple $((a+z)^{-1},0,c)\in\mathcal T_0$ with $a\in\C^*$ and since $h(z)=(a+z)^{-1}$ is invertible and $\rho(L_0)=(a+z)\partial$, it follows immediately that $\rho(L_i)$ given by \eqref{e:Li012} preserves $\C[[z]]$.
\end{proof}

\begin{thm}
	Let ${\mathfrak g}=\Witt, \Vir$. It holds that the set:
	\[ 
	\left\{\begin{gathered}  \rho\in \Hom_{\text{Lie-alg}}( {\mathfrak g} , \Diff( \Czz )) \\  
	\text{ s.t. $\rho(L_0)$ is of first order} \end{gathered} \right\} 
	\]
is in bijection with the triples  $(h(z),b(z),c)\in\mathcal T_1$ where $h(z),b(z)\in  \Czz$ and $h'(z)\in\Czz^*$ through expression~\eqref{e:Lihcb}.
\end{thm}

\begin{proof}
It is a consequence of Theorems~\ref{t:Witt>} and~\ref{t:Witt>-line}.
\end{proof}

%
%


%
%
%

%
\section{Final Remarks}\label{sec:remarks}

Let us finish this paper by addressing how our results connect with a bunch of topics such as enveloping algebras,  the ${\mathcal W}_{1+\infty}$-algebra, simple $\Vir$-modules (\cite{Iohara-Koga}), Virasoro constraints (\cite{Givental}), simple $\sl(2)$-modules (\cite{Block}) and polynomial solutions to differential equations (\cite{Turbiner}). 

\subsection{Enveloping Algebras and ${\mathcal W}$-algebras}\label{subsec:env-alg-w}

In this subsection we apply the above results to show how the universal enveloping algebras relate to the algebra of differential operators. 


\begin{prop}\label{p:UWitt>}
Let $\rho:\Witt_> \to {\mathbb C}[z][\partial]$ be given as in~\eqref{e:Lihcb}. Assume that  $\rho\vert_{\sl(2)}$ is associated to a triple $(h(z)=z,b(z),c)\in\mathcal T_1$ where $b(z)\in \C[z]$. 

If the Casimir operator of $r(\rho)=\rho\vert_{\sl(2)}$ is not $1$, then:
	\[
	\rho:{\mathcal U}(\Witt_>) \to {\mathbb C}[z][\partial] 
	\]
is surjective (here ${\mathcal U}$ denotes the universal enveloping algebra).
\end{prop}

\begin{proof}
Let  us prove that $z^k = h(z)^k \in  \rho({\mathcal U}(\Witt_>))$ for all $k\geq 0$. Let us denote $\rho(L_0)$ by $L$. By \eqref{e:Lihcb}, one has $\rho(L_k)= h(z)^k (L+k c)$,  and taking into account that $L\circ h(z) = h(z)\circ (L-1)$, we get:
	\[
	\begin{aligned}
	\rho(L_{-1})\rho(L_{k+1}) & = 
	h(z)^k (L-c-k-1)(L+(k+1)c)  
	\\
	\rho(L_{0})\rho(L_{k}) & = 
	h(z)^k (L-k) (L+k c)   
	\\
	\rho(L_{1})\rho(L_{k-1}) & = 
	h(z)^k (L+c-k+1) (L+(k-1)c)   
	\end{aligned}
	\]
for any $k\geq 0$. We observe that  $ (L-c-k-1)(L+(k+1)c), (L-k) (L+k c) , (L+c-k+1) (L+(k-1)c)$ are linearly independent as polynomials in $L$ if and only if $c\neq 0,-1$. By Theorem~\ref{t:Un-Env-Al-sl1}, it follows that when the Casimir operator of $\rho\vert_{\sl(2)}$, which  is given by $ (2c+1)^2$, is different from $1$, we get:
	\[ 
	z^k =h(z)^k \,\in \, \rho({\mathcal U}(\Witt_>)) \qquad \forall k\geq 0
	\]
since it can be obtained as a linear combination of the three operators above. Recalling that $\rho(L_{-1})= \partial + b(z)$, one has that $\partial\in  \rho({\mathcal U}(\Witt_>))$, and  the statement  follows. 
\end{proof}


\begin{prop}
Let $\rho:\Witt_> \to {\mathbb C}[z][\partial]$ be given as in~\eqref{e:Li012}. 

If $\rho\vert_{\sl(2)}$ is associated to a triple $(\frac1z,b(z),c)\in\mathcal T_0$, then $b(z)\in \C[z]$ and  $b(0)^2 - b(0)-c(c+1)=0$. 

If the Casimir operator of $\rho\vert_{\sl(2)}$ is not $1$, then $\rho:{\mathcal U}(\Witt_>) \to {\mathbb C}[z][\partial] $ is surjective. 
\end{prop}


\begin{proof}
Note that $\rho(L_0)(\C[z])\subseteq \C[z]$, implies that $b(z)\in \C[z]$. Recalling the expression of the free term of $\rho(L_1)$ from the proof of Theorem~\ref{t:repr-sl2}, one obtains the constraint  $b(0)^2 - b(0)-c(c+1)=0$.
	
Note that $\rho(L_{-1})=z\in \rho({\mathcal U}(\Witt_>))$ and, thus $\C[z]\subseteq \rho({\mathcal U}(\Witt_>))$. It remains to show that for each $k\geq 1$ there exists an operator in $ \rho({\mathcal U}(\Witt_>))$ of order $k$ and symbol $1$. 

If we  denote $\rho(L_0)$ by $L$, recalling the expression~\eqref{e:Li012} for $\rho(L_i)$ and the computations of the proof of Lemma~\ref{l:LiPoch}, we get:
	\[
	\begin{aligned}
	\rho(L_{-1})\rho(L_{k+1}) & = 
	h(z)^k (L-\lambda-k-1)(L+(k+1)\lambda) P(L-\lambda-k,k) 
	\\
	\rho(L_{0})\rho(L_{k}) & = 
	h(z)^k (L-k) (L+k\lambda ) P(L-\lambda-k,k) 
	\\
	\rho(L_{1})\rho(L_{k-1}) & = 
	h(z)^k (L+\lambda-k+1) (L+(k-1)\lambda) P(L-\lambda-k,k) 
	\end{aligned}
	\]
for any $k\geq 0$. We observe that  $ (L-\lambda-k-1)(L+(k+1)\lambda), (L-k) (L+k \lambda) , (L+\lambda-k+1) (L+(k-1)\lambda)$ are linearly independent as polynomials in $L$ if and only if $\lambda\neq 0,-1$. Recalling Theorem~\ref{t:Un-Env-Al-sl1}, it follows that when the Casimir operator of $\rho\vert_{\sl(2)}$, which  is given by $(2c+1)^2=(2\lambda+1)^2$ since either $\lambda=c$ or $\lambda=-c-1$, is different from $1$, we get:
	\[ 
	h(z)^k P(L-\lambda-k,k) \,\in \, \rho({\mathcal U}(\Witt_>)) \qquad \forall k\geq 0
	\]
since it can be obtained as a linear combination of the three operators above. A straightforward computation shows that $h(z)^k P(L-\lambda-k,k)$ is a differential operator  of order $k$ and symbol $1$.
\end{proof}

The same arguments as in  Proposition~\ref{p:UWitt>} prove the following result. 

\begin{prop}\label{p:UWitt}
Let $\rho:\Witt \to {\mathbb C}[z,z^{-1}][\partial]$ be given as in~\eqref{e:Lihcb}  by a triple $(h(z)=z^{\pm 1},b(z),c)\in\mathcal T_1$ where $c\in \C$ and $b(z)\in \C[z,z^{-1}]$.

If the Casimir operator of $\rho\vert_{\sl(2)}$ is not $1$, then:
	\[
	\rho:{\mathcal U}(\Witt) \to {\mathbb C}[z,z^{-1}][\partial] 
	\]
is surjective (here ${\mathcal U}$ denotes the universal enveloping algebra).
\end{prop}

\begin{rem}
Proposition~\ref{p:UWitt} shows that ${\mathbb C}[z,z^{-1}][\partial]$  can be obtained as a quotient of the universal enveloping algebra of the Witt algebra by a certain ideal. This was known for particular choices of $\rho:\mathcal U(\Witt)\to {\mathbb C}[z,z^{-1}][\partial] $ (see \cite{Pope}). It is not difficult to generalize Proposition~\ref{p:UWitt} to  the cases of $\C(z)$ and $\Cz$. Making use of the results of \cite{Plaza-Tejero} as well as our \S\ref{subsec:env}, one can explicitly compute that ideal.
\end{rem}

%



\subsection{${\mathcal W}_{1+\infty}$-algebra} 
The ${\mathcal W}_{1+\infty}$-algebra, also known as $\widehat{\mathcal D}$ in the literature, is the universal central extension of ${\mathbb C}[z,z^{-1}][\partial]$. Indeed from \cite{GF}, it is known that the Lie algebra ${\mathbb C}[z,z^{-1}][\partial]$ has a unique, up to isomorphism, central extension whose cocycle, discovered by Gelfand and Fuks, is given by:
	\[
	\Psi_{\mathrm{GF}}(f(z)\partial^m, g(z)\partial^n) 
	\,:=\, 
	\frac{m! n!}{(m+n+1)!} \operatorname{Res}_{z=0}(\partial^{n+1} f(z)) (\partial^m g(z))d z\, .
	\]

Now, let $\rho:\Witt\to {\mathbb C}[z,z^{-1}][\partial]$ be given by a triple $(h(z),b(z),c)\in\mathcal T_1$ as in \eqref{e:Lihcb}. It follows that $\rho^*\Psi_{\mathrm{GF}}$, which yields a $2$-cocycle of $\Witt$, has to be proportional to the  cocycle $\Psi$ given by equation~\eqref{eq:Vir-co}. Recalling \cite[Theorem 2.14]{Plaza-Sigma}, we obtain:
	\[
	\rho^*\Psi_{\mathrm{GF}} 
	 \,=\, -2(1- 6 c+ 6 c^2)v(h(z)) \Psi
	\,=\, (1- 3\rho( {\mathcal{C}}))v(h(z)) \Psi
	\]
where $v$ is the valuation of $\C[z,z^{-1}]$ defined by $z$ and $\rho({\mathcal{C}})$ is the Casimir operator of $\rho\vert_{\sl(2)}$ (see Theorem~\ref{t:Casimir}). Hence, mapping the central element $K$ of $\Vir$ to the central element of ${\mathcal W}_{1+\infty}$, we conclude that $\rho$ induces a map:
	\[
	\widehat \rho: \Vir \, \longrightarrow\, {\mathcal W}_{1+\infty}\, . 
	\]

Conversely, assume we are given a representation $\tau:\Vir \to {\mathcal W}_{1+\infty}$ that maps the central element to the central element and such that the differential operator part of $\tau(L_0)$ is of order $1$. After quotienting by the central element, we get a map $\bar\tau: \Witt\to {\mathbb C}[z,z^{-1}][\partial]$ that must acquire the form \ref{e:Lihcb}. It is straightforward to check that $\tau= \widehat{\bar\tau}$. 

On the one hand, consider the subalgebra $V'=\C[z,z^{-1}]$ of $V=\C(z)$.  Then, the parameters of Theorem~\ref{t:Witt>} can be chosen such that $\rho(V')\subseteq V'$. For instance, set $(h(z)= z, b(z)=\beta, c=\alpha)\in\mathcal T_1$ with $\alpha,\beta\in \C$. Then, the restriction to $V'$ of the representation $\rho$ defined by \eqref{e:Lihcb} acquires the form:
	\[
	\rho(L_i)(z^n)\,=\, (-z^{i+1}\partial +(\alpha i +\beta)z^i)(z^n)\,=\,  (\alpha i + \beta -n) z^{n+i}
	\]
which is the expression of the $\Vir$-modules of the intermediate series (see \cite[\S1.2.6]{Iohara-Koga}). 

\subsection{Instances of $\Witt_>$ in the literature} 
The differential operators $\{L_i\vert i\geq -1\}$ considered by Givental in \cite[\S3]{Givental} are precisely those defined by the representation $\rho\in  {\mathcal R}_0(\Witt_{>})$ associated to the triple $(z, -\frac12,\frac12)\in\mathcal T_0$ by \eqref{e:Li012}.


In the case $V'=\C(q)$, we set $\lambda=-1$ and $b(q)=q t(q)$ with $t(q)\in \C(q)$ with poles in $\alpha_1,\cdots,\alpha_j$, then $\rho$ induces an action on $\C[q, (q-\alpha_1)^{-1},\ldots, (q-\alpha_j)^{-1}]$. This is the case of \cite[\S7.1]{Block}. 

Similarly, consider the natural embedding of $\C[z]$ in $\C(z)$ or in $\Cz$. If $n$ is a non-negative integer, set $(h(z)=z,b(z)=c,c=\frac{n}2)\in\mathcal T_1$. The first three operators of the representation~\eqref{e:Lihcb} read as follows:
	$$ \rho(L_{-1}) \,= \,   -\partial,\quad \rho(L_0)\,  = \,  -z \partial+ \frac{n}2,\quad \rho(L_1)\,  = \,  -z^2 \partial + n z, $$
and they  generate the Lie algebra $\langle J_n^+, J_n^0, J_n^{-} \rangle$ considered in \cite[Equation~(8)]{Turbiner}.


\subsection{Examples with $n_0>1$}

Let us exhibit instances of representations where $n_0>1$. 

\begin{prop}
Let a polynomial $b(z)\in\C[z]$ and a constant $c\in\C$ be given. Then, the map $\widehat\rho : \Witt_>\to \Diff(\C[q,q^{-1}])$ defined by:
	\[
    \widehat\rho(L_i)= 
    \left(-q^2\partial_q\right)^{i+1}\circ\left(\frac{1}{q}\right)  +\left(b\left(-q^2\partial_q\right) + ic\right)\circ \left(-q^2\partial_q\right)^i
    \, .\]
is a representation of  $\Witt_>$.
\end{prop}

Note that $\widehat\rho(L_i) \in  \operatorname{Diff}(\C[q,q^{-1}] )$  is a differential operator of order $ \operatorname{deg} b(z) +i $. 

\begin{proof}
For the proof, we start with the case of \S\ref{sec:Witt} and mimic  the Fourier transform.  Recall that the Fourier transform of a function $f(z)$ is given  by the formula:
	\[
	\widehat f(\xi)  \,:=\,  \int_{-\infty}^{\infty} f(z)   e^{-2\pi {\mathbf i} z\xi} dz
	\]
and, accordingly, the following expressions hold:
	\begin{align*}
	\left( \frac{d^n}{dz^n} f(z) \right)^{\widehat{\,}} \,&  =\, (2\pi {\mathbf i} \xi)^n \widehat f(\xi)
	\\
	\left( z^n f(z) \right)^{\widehat{\,}} \,&   =\, \left(\frac{{\mathbf i} }{2\pi }\right)^n \frac{d^n \widehat f(\xi)}{d\xi^n}  
	\end{align*}
	
In our setting, let   $q=\frac{-1}{2\pi{\mathbf i}\xi}$ and, accordingly, it holds that  $-q^2\partial_q = \frac{{\mathbf i} }{2\pi }\partial_{\xi}$. Inspired by the Fourier transform, we define a map from differential operators acting on $\C[z]$ to differential operators acting on $\C[q,q^{-1}]$ as follows:
	\begin{align*}
	\operatorname{Diff}(\C[z]) \, \longrightarrow & \, \operatorname{Diff}(\C[q,q^{-1}])
	\\
	\partial^n \, \longmapsto & \, \left(\frac{-1}{q}\right)^n 
	\\
	z^n   \, \longmapsto & \,  (-q^2\partial_q)^n 
	\end{align*}
that is a homomorphism of Lie algebras.

Applying this transformation to the representation $\rho$ on $\C[z]$, given by the expression \eqref{e:Lihcb} associated to a triple $(h(z)=z, b(z),c)\in\mathcal T_1$, we obtain the claimed $\widehat\rho$. 
\end{proof}

\begin{exam}
Computing $\rho(L_i)$ ($i=-1,0,1$) for $\rho$ associated to the triple $(h(z)=z, b(z)=-c,-c)\in\mathcal T_1$ given by \eqref{e:Lihcb}, we obtain:
	\begin{align*}
	\rho(L_{-1}) = -\partial  \qquad & \qquad \widehat\rho(L_{-1}) =  q^{-1}
	\\
	\rho(L_{0}) = -z \partial - c   \qquad & \qquad  \widehat\rho(L_0) =  -q\partial_q - c +1
	\\
	\rho(L_{1}) = -z^2\partial - 2c z \qquad & \qquad \widehat\rho(L_1) =   q^3\partial_q^2 + 2c q^2 \partial_q
	\end{align*}
that corresponds to  the triple $(h(q)=q, b(q)= 1-c,\lambda)\in\mathcal T_0$ in equation~\eqref{e:Li012} where $\lambda\in\{-c, c-1\}$ is the unique element  that belongs to $\C_{\mathrm{sl}}$. This relation has already appeared in the literature on Conformal Field Theory;  see, for instance,  \cite{Ribault}, equations~2.3.5 and 4.2.23.
\end{exam}

Let us exhibit a general method to produce new representations out of those with $n_0=1$. Recall from \cite{Dixmier} that the automorphism group of the Weyl algebra $\C[z][\partial]$ is generated by $\Phi_{n,\alpha}, \Phi'_{n,\alpha}$ defined by:
	\[
	\begin{aligned}
	\Phi_{n,\alpha}(z) &\,=\, z 
	\\
	\Phi_{n,\alpha}(\partial) &\,=\, \alpha z^n + \partial
	\end{aligned}
	\qquad \qquad
	\begin{aligned}
	\Phi'_{n,\alpha}(z) &\,=\, z + \alpha \partial^n
	\\
	\Phi'_{n,\alpha}(\partial) &\,=\,  \partial
	\end{aligned}
	\]
where $n\in \Z_{\geq 0}$ and $\alpha\in \C$. 	As long as an automorphism $\Phi$ is defined on the image of $\rho:\Witt_>\to \C[z,z^{-1}][\partial]$, it holds that $\Phi\circ \rho$ is another representation. 

\begin{exam}
Let $\rho$ be as in~\eqref{e:rhoexp} defined by  $(h(z),  b(z), c)\in \mathcal T_0$ where:
	\[
	h(z)=\frac1z \quad , \quad 
	b(z)=\frac{1-d}2
	\quad  \quad    
	\]
with  $d\in \C$ and $c$ is equal to the unique element of $\{\frac{d-1}{2}, -\frac{1+d}{2}\}$ that belongs to $\C_{\mathrm{sl}}$. Let us consider $\Phi'_{1,-1}$. Then, it holds that:
	\[
	\begin{aligned}
	( \Phi'_{1,-1} \circ\rho)(L_{-1})  &\,=\, - \partial + z, 
	\\
	( \Phi'_{1,-1} \circ\rho)(L_{0})  &\,=\,  -\partial^2+z \partial + \frac{1-d}2 ,
	\\
	( \Phi'_{1,-1} \circ\rho)(L_{1})  &\,=\,   -\partial^3 + z \partial^2 +(1-d)\partial. 	\end{aligned}
	\]
This is the subalgebra of  $\C[z][\partial]$ considered in \cite[Equation~8.8]{Mulase}.
\end{exam}


\end{document}